\documentclass[11pt]{amsart}

\newcommand{\ds}{\displaystyle}
\newcommand{\cu}{\mathcal{K}}
\newcommand{\ca}{{\mathcal{K}}_0}
\newcommand{\co}{{\mathcal{K}}_{\hbox{reg}}}

\usepackage{color}
\usepackage{enumerate}

\begin{document}

\bibliographystyle{plain}

\newtheorem{theorem}{Theorem}[section]
\newtheorem{proposition}{Proposition}[section]
\newtheorem{lemma}{Lemma}[section]
\newtheorem{definition}{Definition}[section]
\newtheorem{corollary}{Corollary}[section]
\newtheorem{remark}{Remark}[section]
\title
{{\bf{ Centro-Affine Invariants for Smooth Convex Bodies}} }
\author{Alina Stancu${}^{1}$}

\thanks{${}^{1}$ Partially supported
by an NSERC grant.}

\maketitle

\begin{abstract}
Employing a centro-affine flow on smooth convex bodies, we generate
new centro-affine differential invariants. One class of the newly
defined invariants is the object of a sharp isoperimetric
inequality, while other new inequalities on known centro-affine
invariants are obtained as a byproduct of the flow's study.
Furthermore, this approach led to a geometric interpretation of the
$L_{\phi}$ affine surface area recently introduced by Ludwig and
Reitzner.
\end{abstract}

\section{Introduction}

Lutwak's seminal work placing the Brunn-Minkowski theory in the
larger context of the Brunn-Minkowski-Firey theory of convex bodies
\cite{Lutwak1}, \cite{Lutwak2} had an impressive outcome on
centro-affine invariants. The new view brought a renewed focus on
the class of convex bodies (compact convex sets) containing the
origin in their interior on which many affine invariants and many
affine-invariant inequalities were derived, \cite{Chou_Wang},
\cite{Cianchi}, \cite{Haberl_Ludwig}--\cite{Haberl_Lutwak},
\cite{Ludwig1}, \cite{Lutwak_Oliker}--\cite{LYZ2},
\cite{Werner}--\cite{zhang}. More importantly, these affine
inequalities were employed successfully in problems apparently
unrelated.

Historically, after Felix Klein outlined his Erlangen Program in
1872, the main period of intense activity in the study of geometric
invariants under $SL(n)$ and $SA(n)$, special affine group of
transformations, was due to Blaschke and his school of differential
geometry. The results were inherently assuming certain regularity or
smoothness of convex bodies. Many of those assumptions were removed
later in time. See, for example, the case of the celebrated affine
surface area introduced in the '20's by Blaschke for sufficiently
smooth convex surfaces  and extended to arbitrary convex bodies more
than fifty years later \cite{Petty}, \cite{L1}, \cite{Lu2},
\cite{SW1}. Unfortunately for many problems in approximation theory,
the generalized affine surface area is zero for most convex bodies,
\cite{gruber3}.

For a sufficiently smooth convex body $K \subset \mathbb{R}^n$, the
affine surface area, $\Omega (K)$, and the affine isoperimetric
inequality relating it to the volume of the convex body by
$\displaystyle \Omega^{n+1}(K) \leq \omega_n^2 n^{n+1} Vol
(K)^{n-1}$, $\omega_n=Vol (B^n)$
has proved to be an invaluable tool in convexity (see the monograph
\cite{leichtweiss}), stochastic geometry, \cite{Bar}, \cite{gruber2}
-- \cite{Haberl_Ludwig}, \cite{reitzner}, differential geometry and
differential equations \cite{andrews}, \cite{andrews2},
\cite{sapiro}, \cite{TW-1} -- \cite{W}.

Since many basic problems in the aforementioned fields  are
equi-affine invariant, it is interesting to concentrate on
$SL(n)$-invariants and, in fact, equi-affine surface areas have
found numerous applications there. Using the theory of valuations on
convex bodies, Ludwig and Reitzner \cite{Ludwig-Reitzner} were able
to characterize a rich family of affine surface areas. Classical
equi-affine and centro-affine surface areas as well as, for real $p
>0$, the  $L_p$  affine surface areas, $\Omega_p(K)$, (introduced by
Lutwak \cite{Lutwak2}) belong to this family of affine surface
areas. In a recent follow-up paper, Ludwig extended the definition
of those equiaffine surface areas to include the  $L_p$  affine
surface areas for non-positive $p$ (which is now infinite on
polytopes)  and, in fact, she conjectured that {\em any} equiaffine
area is an $L_p$ affine surface area, \cite{Ludwig}.

Studying affine and centro-affine invariants on sufficiently regular
convex bodies can still shed light on the behavior of arbitrary
convex bodies in many respects. Although things seems simpler for,
say, polytopes for which affine surface area is null, it is
precisely on that class of bodies where many things are still not
understood. One may recall that a major affine invariant open
problem, Mahler's conjecture, which claims the minimum of the volume
product on simplices, resisted many attempts to be solved except in
several cases (see \cite{schneider} for an overview). Here, we
 start a study of centro-affine invariants on the class of smooth convex bodies. The heuristic idea
to derive new centro-affine invariants is as follows. Define a
centro-affine flow on this class, thus essentially define a
one-parameter family of convex bodies which commutes with
centro-affine transformations. Consequently, along the flow,
derivatives of any centro-affine invariant quantities associated to
a convex body with respect to this parameter generate new
centro-affine invariants. It models in a way a property of planar
curves in affine geometry stating that any local affine invariant of
a smooth curve is the affine curvature or derivatives of the affine
curvature with respect to the affine arclength parameter,
\cite{guggenheimer}.

The paper starts with the definition of the centro-affine flow
mentioned above and several of its properties. We then recall the
definition of $L_p$ affine surface areas and list several
inequalities on $\Omega_p$ implied by the flow. The fourth section
states the main result defining the new centro-affine invariants and
proves a sharp isoperimetric inequality for a class of such
invariants. In the last section, we provide a geometric
interpretation of the $L_{\phi}$ affine surface area introduced by
Ludwig and Reitzner.

\section{Curvature Flows: Short Time Existence}

Let $K$ be a smooth convex body in $\mathbb{R}^{n}$, containing the
origin in its interior, and having positive Gauss curvature at all
points of its boundary. Following H\"ormander, we denote the set of
all such convex bodies  by $\co$, \cite{Hor}. We will usually
identify $K \in \co$ with a smooth embedding $X_K: \mathbb{S}^{n-1}
\rightarrow \mathbb{R}^n$ whose image in $\mathbb{R}^n$ is $\partial
K$ and the unit normal vector at $X_K(u)$ is precisely $u$. Note
that our choice of normal points outward.

At each point of $\partial K$, one has a well defined Gauss
curvature $\cu$, a centro-affine curvature $\ca$, and an {\it{affine
normal vector}} $\mathcal{N}$, respectively. All of them are viewed
here as functions on the unit sphere $\mathbb{S}^{n-1}$. In
particular, recall that  the centro-affine curvature is defined by
the formula
\begin{equation}
\ca (u)=\frac{\cu (u)}{\left(X_K (u) \cdot u \right)^{n+1}},\ \ \
\forall u \in \mathbb{S}^{n-1}, \label{eq:cac}
\end{equation}
where $\cdot$ is the usual scalar product of $\mathbb{R}^n$.
Geometrically, the centro-affine curvature is
is related to a normalized volume of the centered osculating
ellipsoid $E_K(p)$ at $p:=X_K(u)$ by

\begin{equation}
\ca(u)=\frac{Vol (\mathbb{B}^n)}{\left[ Vol (E_K (p) \right]^2}.
\end{equation}

This view of $\ca$ emphasizes its centro-affine invariance.

The affine normal vector field $\mathcal{N}$ is, at each point of
$\partial K$, transverse in the ambient space to the tangent space
of $\partial K$, but not, in general, normal to it in the Euclidean
sense. Its precise analytic definition is

\begin{equation}
{\mathcal{N}}= \frac{1}{n}\, \Delta_g X, 
\label{eq: affine_normal}
\end{equation}

where $\Delta_g$ is the Laplace-Beltrami operator on the strictly
convex hypersurface $\partial K$ with the respect to the (affine
invariant) Blaschke metric applied to each component of $X$. We will
come back to this metric later as for now, besides ${\mathcal{N}}$'s
invariance under affine transformations, we will need solely its
(Euclidean) normal component

\begin{equation}
\mathcal{N}(u) \cdot u =- {\cu}^{\frac{1}{n+1}}. \label{eq:normal}
\end{equation}

For more details on this object, we refer the reader to
\cite{leichtweiss} and \cite{li-simon-zhao}.

\begin{theorem}
Let $K$ be a convex body belonging to $\co$ and let $p$ be a real
number $p \neq -n$ or $ 0$. Then there exists a short time solution
to the flow

\begin{equation} \left\{  \begin{alignedat}{2}
\frac{\partial X(u,t)}{\partial t}
&=&{\hbox{sgn}}\left({\frac{p}{n+p}}\right)\:
{\ca}^{\frac{n(p-1)}{(n+1)(n+p)}} (u,t)  \: \mathcal{N}(u,t)
\\
X(u,0)&=&X_K(u) \hspace{5cm}  \end{alignedat}\right.
\label{eq:flows} \end{equation}

in $\co$. \label{theorem:STE}
\end{theorem}

\begin{remark}
A solution in $\co$ to any of the $p$-curvature flows above on some
non-empty time interval $[0,T)$ is a class of embeddings
$\{X(u,t)\}_{t \in [0, T)}$ satisfying (\ref{eq:flows}) whose images
in $\mathbb{R}^n$ are convex bodies in $\co$. Often the parameter
$t$ is referred to as time and, consequently, we talk about a time
evolution of quantities which describe the convex body solution.

In other words, one looks at a deformation of $K \in \co$ along a
vector field transverse to its boundary. The claim of Theorem
\ref{theorem:STE} is that the resulting compact sets remain, at
least for some time, convex bodies in $\co$.
\end{remark}

\begin{remark}
Additionally, note that $\forall A \in SL(n)$, we have
(\ref{eq:flows}) equivalent with

\begin{equation} \left\{  \begin{alignedat}{2}
\frac{\partial \tilde{X}(u,t)}{\partial t}
&=&{\hbox{sgn}}\left({\frac{p}{n+p}}\right)\:
{\ca}^{\frac{n(p-1)}{(n+1)(n+p)}} (u,t)  \:
{\tilde{\mathcal{N}}}(u,t)
\\
\tilde{X}(u,0)&=&\tilde{X}_K(u), \hspace{5cm}
\end{alignedat}\right. \label{eq:Aflows} \end{equation} where
$\tilde{X} (u, t)=A \cdot X (u,t)$ and $\tilde{{\mathcal{N}}} (u,
t)=A \cdot {\mathcal{N}} (u,t)$, ($\ca$ is centro-affine invariant).
\end{remark}

\smallskip

\begin{proof}[{\bf{Proof of Theorem \ref{theorem:STE}}}]
We will first use (\ref{eq:normal}), and simultaneously
(\ref{eq:cac}), to describe the deformation of the the flows in the
direction of the Euclidean normals, and in terms of Euclidean
quantities, namely,

\begin{equation} \left\{  \begin{alignedat}{2}
\frac{\partial X(u,t)}{\partial t}
&=&-{\hbox{sgn}}\left({\frac{p}{n+p}}\right)\:  (X(u,t) \cdot
u)^{-\frac{n(p-1)}{n+p}} \: {\cu}^{\frac{p}{n+p}}(u, t)\, u
 \\
X(u,0)&=&X_K(u). \hspace{6.5cm} \end{alignedat}\right.
\label{eq:fwsn} \end{equation}

It's been noted in the theory of geometric evolution equations that
only the normal component of the velocity affects the shape of the
resulting hypersurfaces, while the tangential component is a mere
re-parametrization, thus (\ref{eq:fwsn}) and (\ref{eq:fws}) are
equivalent. The images of solutions are identical even if the
solutions may differ analytically as they may represent different
parameterizations of the same convex body $K(t)$ in $\co$. It is
therefore sufficient to prove short time existence of solutions to
(\ref{eq:fwsn}).

Moreover, we will pass to a scalar form of the flow which will
emphasize its strict parabolicity. This is accomplished via the
support function of the evolving convex bodies $K(t)$ with $K(0)=K$.
Recall that for an arbitrary convex body $K$, its support function,
as a function on the unit sphere, is

\begin{equation}
h_K: \mathbb{S}^{n-1} \to \mathbb{R}, \ \ \ h_K(u)=\sup \{x \cdot u
\mid x \in K\}, \label{eq:support}
\end{equation}

and that $K$ is recovered uniquely by $$K=\{x \in \mathbb{R}^n \mid
x \cdot u \leq h(u) \ \ {\hbox{for all}}\ u \in \mathbb{S}^{n-1}
\}.$$

Consequently, we have $h_K(u)= X(u) \cdot u$, in all unitary
directions $u$, and the smoothness of $\partial K$ is equivalent to
the smoothness of the support function of $K$. See \cite{schneider}
for a complete treatment of the support function. In particular, we
will use here an important feature of the support function of convex
bodies in $\co$ by relating it to the Gauss curvature of their
boundaries. Pointwise, for all $u \in \mathbb{S}^{n-1}$, the Gauss
curvature of $\partial K$ is related to the support function of the
convex body by

\begin{equation}
\frac{1}{\cu}=\det (\bar{\nabla}^2 h + {\hbox{Id}} \, h),
\label{eq:curv_h}
\end{equation}

where $\bar{\nabla}$ is the covariant derivative  on
$\mathbb{S}^{n-1} $ endowed with an orthonormal frame.

By using the notation $h(u,t)$ for $h_{K(t)} (u)$, we re-write
(\ref{eq:fwsn}) as

\begin{equation} \left\{  \begin{alignedat}{2}
\frac{\partial h(u,t)}{\partial t}
&=&-{\hbox{sgn}}\left({\frac{p}{n+p}}\right)\: \left(h(u,t)
\right)^{-\frac{n(p-1)}{n+p}} \: k^{\frac{p}{n+p}}
\\
h(u,0)&=&h_K(u). \hspace{4.5cm}  \end{alignedat}\right.
\label{eq:fwsh} \end{equation}

Note that, for $p=1$, the Cauchy problem (\ref{eq:fwsh}) is the
well-known affine curvature flow, \cite{andrews}, \cite{sapiro},

\begin{equation}
h_t= -k^{\frac{1}{n+1}},\ \ h(u,0)=h_K(u).
\end{equation}

Here and thereafter we stop writing the argument $u$ unless there is
a potential risk of confusion.

Consider the operator

\begin{equation}
\displaystyle L(h)=-{\hbox{sgn}}\left({\frac{p}{n+p}}\right)\:
\left(h(u,t) \right)^{-\frac{n(p-1)}{n+p}} \: \left[\det
(\bar{\nabla}^2 h + {\hbox{Id}} \, h)\right]^{-\frac{p}{n+p}},
\end{equation}

where $\bar{\nabla}$ is the covariant derivative on the unit sphere
as above. We want to linearize $L$ over the convex hypersurface
given by $h$, thus we will compute $\displaystyle \delta L= \left\{
\frac{d}{d\delta}L(h^{\delta})\right\}_{{\mid}_{\delta =0}}$, where
$h^{\delta}$ is a small perturbation of $h$ at some fixed time with
$h^0=h$.

 Let $\displaystyle S$ be the determinant of the matrix $\tilde{S}$ with entries $\displaystyle a_{ij}=h_{ij} + \delta_{ij} \, h$ where we denoted the covariant derivatives corresponding to an orthonormal frame on the unit sphere as the usual differentiation.
By a straightforward computation, we have the following lemma.

\begin{lemma} Let $\displaystyle h^{\delta}$ as above and let
 $\phi : \mathbb{S}^{n-1} \times [0,
t_0) \to \mathbb{R}$ be defined by $\displaystyle \phi =  \left\{
\frac{d}{d\delta} (h^{\delta}) \right\}_{{\mid}_{\delta =0}}$. Then
$\phi$ satisfies the following differential equation

\begin{eqnarray}
 \frac{\partial \phi}{\partial t}&=& \mid \frac{p}{n+p}
\mid k^{\frac{n+2p}{n+p}} h^{-\frac{n(p-1)}{n+p}}\, \sum_{i,j=1}^n
\frac{\partial S}{\partial a_{ij}} \: \left(\phi_{ij} + \delta_{ij}
\, \phi \right)  \label{eq:lin}  \\   & + &\mid \frac{p}{n+p} \mid
\frac{n(1-p)}{n+p} \, k^{\frac{p}{n+p}} h^{-\frac{p(n+1)}{n+p}}\:
\phi .   \nonumber
\end{eqnarray}

In particular, the linearization of (\ref{eq:fwsh}) about the unit
sphere is

\begin{equation}
 \frac{\partial \phi}{\partial t}=\, \mid \frac{p}{n+p}
\mid (\Delta_{\mathbb{S}^{n-1}} \phi + (n-1) \phi ).
\end{equation}
\label{lemma:lin}
\end{lemma}

Recall that $S$ is the $(n-1)$-symmetric polynomial
  on the eigenvalues of the matrix $\tilde{S}$, and these are the principal radii of curvature at the boundary. It has been shown in \cite{Mi} that $\displaystyle \dot{S}:=\left(\frac{\partial S}{\partial a_{ij}} \right)_{ij}$ is a positive definite bilinear form as long as the second fundamental form of $\partial K$ is positive definite, hence as long as $K$ is strictly convex.
It follows form the last lemma that, for a convex body $K \in \co$,
the equation (\ref{eq:lin}) is strictly parabolic. Thus \cite{Lady}
implies the short time existence of solutions to (\ref{eq:fwsh})
and, consequently, to (\ref{eq:flows}).

\end{proof}

We conclude the introduction of the flow with two of its properties.

\begin{proposition}{Containment Principle.} If $K_{in}$ and $K_{out}$ are two convex bodies
in $\co$ such that $K_{in} \subset K_{out}$, and $p$ is a real
number other than $-n$ or zero, then $K_{in}(t) \subseteq
K_{out}(t)$ for as long as the solutions $K_{in}(t)$ and
$K_{out}(t)$ to (\ref{eq:flows}) (with given initial data
$K_{in}(0)=K_{in}$, $K_{out}(0)=K_{out}$) exist in $\co$.
\label{prop:containment} \end{proposition}

\begin{proof} Consider first the extinction case $\alpha:=p/(n+p) >0$ in which the convex bodies
shrink. Suppose that two bodies in $\co$ are in the relation $K_{in}
\subset K_{out}$ at time $t=0$, while at some later time they
intersect. Let $t_0$ be a first time when $K_{in}(t) \cap K_{out}
(t) \neq \emptyset$. Thus, unless $K_{in}(t)=K_{out}(t)$, there
exists a tangency point between them such that the outer normal at
the tangency point coincides, we call it $u$, and $h_{in} (u,
t_0)=h_{out}(u, t_0)$. Moreover, if $K_{in}(t_0) \neq K_{out}(t_0)$,
hence at the tangency point their Gauss curvatures are in the
relation $\displaystyle \cu_{in} (u, t_0) > \cu_{out} (u, t_0)$,
where strict convexity implied strict inequality. (We are not
claiming that at all tangency points the latter inequality holds,
but that there exists a tangency point with this property.)

Thus  \begin{eqnarray} \frac{\partial h_{in}}{\partial t}(u,t_0)
=-{\hbox{sgn}}\left({\frac{p}{n+p}}\right)\: \left(h_{in}(u,t_0)
\right)^{-\frac{n(p-1)}{n+p}} \: {\cu}_{in}^{\frac{p}{n+p}}(u, t_0)
\\ < -{\hbox{sgn}}\left({\frac{p}{n+p}}\right)\:
\left(h_{out}(u,t_0) \right)^{-\frac{n(p-1)}{n+p}} \:
{\cu}_{out}^{\frac{p}{n+p}}(u, t_0) =\frac{\partial
h_{out}}{\partial t}(u,t_0) \leq 0, \nonumber
\end{eqnarray}
hence the evolving boundaries become again disjoint as the inward
speed of the inner body is greater than the inward speed of the
outer body.

The case $\alpha:=p/(n+p) <0$ leads to the same inequality above,
except for the sign of the velocities $$0 \leq \frac{\partial
h_{in}}{\partial t}(u,t_0) <  \frac{\partial h_{out}}{\partial
t}(u,t_0).$$ Hence, this time the flows are expanding the bodies and
the inner convex body flows outward with a smaller speed than the
outer one, implying the same conclusion.

\end{proof}

To state the next result, let $K^{\circ}$ denote the dual polar body
associated to $K$ with respect to the origin

\begin{equation}
K^{\circ} = \{ y \in \mathbb{R}^n \mid x \cdot y \leq 1,\ \forall x
\in K \}.
\end{equation}

As the origin belongs to $Int(K)$, the polar body is also a convex
body, in fact $K^{\circ} \in \co$, \cite{Hor}. Thus the volume of
$K$, respectively $K^{\circ}$, as compact convex sets of
$\mathbb{R}^n$ with the Lebesgue measure is, respectively, finite,
\cite{schneider}.

\begin{proposition}
If $K \in \co$ is not an ellipsoid centered at the origin, the
$p$-affine flow (\ref{eq:fwsn}) increases the volume product $Vol
(K) \cdot Vol (K^{\circ})$ for as long as the flow exists in $\co$.
Any ellipsoid centered at the origin, flows homothetically under
(\ref{eq:fwsn}), hence the volume product $Vol (K) \cdot Vol
(K^{\circ})$ remains constant until $K$ shrinks to a point in finite
time. \label{prop:vol_pr}
\end{proposition}

\begin{proof}
 Let $K (t)$ evolve by the flow with $K(0)=K$. Meanwhile, recall that

\begin{equation}
Vol (K) = \frac{1}{n} \: \int_{\mathbb{S}^{n-1}} h \, \frac{1}{\cu}
\, d\mu_{\mathbb{S}^{n-1}}(u) =\frac{1}{n} \int_{\partial K} h (\nu
(x))\, d\mu_{K}(x), \label{eq:volume}\end{equation}
 where $\displaystyle d\mu_{ \mathbb{S}^{n-1}}$ denotes the surface area measure of $\displaystyle { \mathbb{S}^{n-1}}$ with the induced metric from $\mathbb{R}^n$ and $\displaystyle d\mu_K$ for the surface area measure of $K$, while $\nu : \partial K \to \mathbb{S}^{n-1}$ is the Gauss map associating to each point $x \in \partial K$ the normal $u$  to $\partial K$ at $x$.

Moreover,
\begin{equation}
Vol (K^{\circ})(t)= \frac{1}{n} \: \int_{\mathbb{S}^{n-1}}
\frac{1}{h^n (u,t)} \, d\mu_{\mathbb{S}^{n-1}}(u),
\end{equation}
thus

\begin{eqnarray}
\frac{d}{dt} (Vol (K(t)) \cdot Vol (K^{\circ}(t))) \hspace{5cm}
\label{omegaP}\\=-{\hbox{sgn}}\left( \frac{p}{n+p} \right) \, \left[
Vol(K^{\circ}) \cdot  \int_{\mathbb{S}^{n-1}}
h^{-\frac{n(p-1)}{n+p}} {\cu}^{-\frac{n}{n+p}}\,
d\mu_{\mathbb{S}^{n-1}} \right. \nonumber \\  \left. - Vol (K) \cdot
\int_{\mathbb{S}^{n-1}} h^{-n-1}\, h^{-\frac{n(p-1)}{n+p}}\,
{\cu}^{\frac{p}{n+p}}\, d\mu_{\mathbb{S}^{n-1}} \right]. \nonumber
\end{eqnarray}

Furthermore we recall a generalized H\"{o}lder inequality introduced
by Andrews, \cite{andrews1}, which we will use to conclude that the
rate of change above is non-negative. If $M$ is a compact manifold
with a volume form $d\omega$, $g$ is a continuous function on $M$
and $F$ is a decreasing real, positive function, then
\begin{equation} \frac{\displaystyle \int_{M} gF(g)\,
d\omega}{\displaystyle \int_{M} F(g)\, d\omega} \leq
\frac{\displaystyle \int_{M} g\, d\omega}{\displaystyle \int_M
d\omega}. \label{eq:Andrews}
\end{equation} If $F$ is strictly decreasing, then equality occurs
if and only if $g$ is constant. Similarly, if $F$ is an increasing
real, positive function, the conclusion holds with $\geq$ in
(\ref{eq:Andrews}) and, similarly, if $F$ is strictly increasing,
then equality occurs if and only if $g$ is a constant function.

We conclude the first statement by re-arranging the terms above and
taking $g=h^{n+1}/\cu$, the reciprocal of the centro-affine
curvature, $F(x)=x^{-p/(n+p)},\ x>0,$ and $d\omega=h^{-n}\,
d\mu_{\mathbb{S}_{n-1}}$ in (\ref{eq:Andrews}).

The second statement is immediate as the centro-affine curvature is
constant for ellipsoids centered at the origin, see for example
\cite{petty}.
\end{proof}

\medskip

Proposition \ref{prop:containment} provides a natural way to extend
these flows outside the space of $\co$ convex bodies via generalized
(weak) solutions. See, for example, similar techniques of
Daskalopoulos-Sesum \cite{dask-sesum}.  Proposition
\ref{prop:vol_pr} suggests that each of the $p$-flows evolves any
initial convex body to an ellipsoid, up to a volume normalization.
However, we will not address here neither the extension problem, nor
the question of long time existence, even if we believe that these
are interesting questions on their own. Equally interesting is the
asymptotic behavior of the solutions as their affine invariant
character suggests subsequential time convergence to ellipsoids as
mentioned above.

Instead,  we will exploit the affine invariant nature of these flows
to derive new centro-affine invariants for smooth convex bodies and
affine inequalities.

\section{$L_p$-Affine Surface Areas}

Let $K$ be an arbitrary convex body. It is well known that the
volume of $K$, $Vol (K)$, understood as the volume enclosed by
$\partial K$ is the simplest affine invariant, at least after its
Euler characteristic which is trivially affine invariant. On the
other hand, it is easy to see that the volume of the convex
hypersurface $\partial K$, which we refer to as the surface area of
$K$ is not affine invariant.

However, an invariant {\em volume} of the hypersurface is the
celebrated affine surface area. This is  the volume of $\partial K$
with respect to the Blaschke metric $\displaystyle d
\sigma={\cu}^{\frac{1}{n+1}} \: d\mu_K$ mentioned earlier,

\begin{equation}
 \Omega (K) = \int_{\partial K} d \sigma = \int_{\partial K}{\cu}^{\frac{1}{n+1}} \: d\mu_K= \int_{\mathbb{S}^{n-1}} {\cu}^{-\frac{n}{n+1}}\, d\mu_{ \mathbb{S}^{n-1}}
\end{equation}

to give two more of its analytical descriptions, the latter valid
for bodies whose Gauss curvature is positive
$\mu_{\mathbb{S}^{n-1}}$-almost everywhere. The literature on the
affine surface area and, more importantly, its applications are
widespread, see in particular the monograph \cite{leichtweiss} and
the paper by Ludwig and Reitzner \cite{LuR}.

An extension of the affine surface area defined by Lutwak
\cite{Lutwak2} in the context of the Firey-Brunn-Minkowski theory of
convex bodies, called the $p$-affine surface area,

\begin{equation}
\Omega_p (K) = \int_{\partial K}  \ca^{\frac{p}{n+p}}\, d\mu_{cK},
\label{eq:p_affine_surface}
\end{equation}
 is a centro-affine invariant of $K$, convex body containing the origin in its interior,
 and this invariant reduces to the usual affine surface area for $p=1$. Note that for $p=0$
 and $p=\pm \infty$, the $p$-affine surface area is $n\, Vol (K)$, respectively,
 $n\, Vol (K^{\circ})$, while for $p=-n$, the affine surface area does not exist.
 Shortly after its definition, this new invariant became the object of many inequalities,
 like \cite{LYZ},\cite{Meyer_Werner}, \cite{ye-werner}, \cite{zhang}. At the core of the
 centro-affine geometry, lies the $p$-affine isoperimetric inequality due to Lutwak, $p \geq 1$,
 which generalizes the classical $p=1$ case

 \begin{equation}
\Omega_p^{n+p}(K)\leq n^{n+p}\omega_n^{2p} Vol^{n-p} (K),
 \end{equation}

 with equality if and only if K is an ellipsoid, \cite{Lutwak2}. Here $\omega_n$ is the volume of the unit ball in $\mathbb{R}^n$.
Many of the $p$-isoperimetric inequalities, in particular the ones
derived by Lutwak, were extended to all $p$'s, see \cite{ye-werner}.

 In fact, the generalized H\"older inequality from the proof of Proposition
 \ref{prop:vol_pr}
  can be stated in a greater generality with the newly introduced $p$-affine surface area as follows.

 \begin{theorem}
 Let $K$ be a convex body with boundary of class $C^2$ and everywhere positive Gauss curvature, $0 \in Int (K)$. Then, for any $p \neq 0, -n$, letting $\displaystyle SGN_p :={\hbox{sgn}} \left( \frac{p}{n+p} \right)$, we have

\begin{equation}
{\hbox{SGN}}_p  \, \frac{\Omega_p
(K)}{\Omega_{-\frac{np}{n+2p}}(K^{\circ})} \leq {\hbox{SGN}}_p \,
\frac{Vol (K)}{Vol (K^{\circ})} \leq {\hbox{SGN}}_p  \,
\frac{\Omega_{-\frac{np}{n+2p}} (K)}{\Omega_{p}(K^{\circ})},
\label{eq:double}
\end{equation}

 with equalities if and only if $K$ is an ellipsoid centered at the origin.

 \end{theorem}

 \begin{proof}
 Consider $n/(n+p)>0$, the other case being analogous. The non-decrease of the volume product, implies from (\ref{omegaP})

\begin{equation}
 \Omega_p (K)  \leq \frac{Vol (K)}{Vol(K^{\circ})} \cdot \int_{\mathbb{S}^{n-1}} h^{-n-1}\, h^{-\frac{n(p-1)}{n+p}}\, {\cu}^{\frac{p}{n+p}}\, d\mu_{\mathbb{S}^{n-1}}.
\end{equation}

However, the last factor is
\begin{equation}
\int_{\mathbb{S}^{n-1}} h^{-n-1-\frac{n(p-1)}{n+p}}\,
{\cu}^{\frac{p}{n+p}}\, d\mu_{\mathbb{S}^{n-1}}= \int_{\partial K}
{\ca}^{\frac{n+2p}{n+p}}\, d\mu_{cK}=\Omega_{-\frac{n^2+2pn}{p}}
(K).
 \end{equation}
 We'll now use that for $C^2_+$ convex bodies containing the origin, and for any $q \neq -n$, we have $\displaystyle
 \Omega_q(K)=\Omega_{n^2/q} (K ^\circ)$. This is known for a larger class of convex bodies, if $q>0$, see \cite{hug}, \cite{Ludwig-Reitzner}.  The extension to all $q \neq -n$ for $C^2_+$ bodies follows from the relationship between the boundary structure of $K$ and that of $K^{\circ}$, see \cite{hug}, \cite{Ludwig}.
 Thus,  we conclude the left inequality in (\ref{eq:double}), which applied to $K^{\circ}$ implies the upper bound on the ratio $\displaystyle Vol (K) / Vol (K^{\circ})$.
 \end{proof}

 An entire collection of inequalities can be derived from the previous theorem. For the rest of this section, we'll restrict our attention to some of them.

  \begin{corollary}
 Let $K$ be a convex body of class $C^2_+$ and let $p$ be a real number.

 If $p \in (-\infty, -n) \cup (0, +\infty)$, then for any ellipsoid $E \subset K$, we have

 \begin{equation}
  \frac{\Omega_{-\frac{np}{n+2p}} (K)}{\Omega_{p}(K^{\circ})}  \geq \,  \frac{\Omega_{-\frac{np}{n+2p}} (E)}{\Omega_{p}(E^{\circ})},
 \end{equation}
 while, for any ellipsoid $E \supset K$,  we have
\begin{equation}
  \frac{\Omega_{p} (K)}{\Omega_{-\frac{np}{n+2p}}(K^{\circ})}  \leq \, \frac{\Omega_{p} (E)}{\Omega_{-\frac{np}{n+2p}}(E^{\circ})}.
 \end{equation}

 If $p \in (-n, 0)$, then for any ellipsoid $E \subset K$, we have

\begin{equation}
  \frac{\Omega_{p} (K)}{\Omega_{-\frac{np}{n+2p}}(K^{\circ})}  \geq \, \frac{\Omega_{p} (E)}{\Omega_{-\frac{np}{n+2p}}(E^{\circ})},
 \end{equation}
 while, for any ellipsoid $E \supset K$,  we have
 \begin{equation}
  \frac{\Omega_{-\frac{np}{n+2p}} (K)}{\Omega_{p}(K^{\circ})}  \leq \,  \frac{\Omega_{-\frac{np}{n+2p}} (E)}{\Omega_{p}(E^{\circ})}.
 \end{equation}

 \end{corollary}

 Furthermore, if $E$ is the John ellipsoid associated to a convex body $K$, it is known that $E \subset K \subset nE$, \cite{schneider}. Applying the previous corollary, we derive the following result.

 \begin{corollary}
Let $K$ be a convex body of class $C^2_+$ and let $p$ be a real
number.

 Then
 \begin{equation}
 \Omega_{-\frac{np}{n+2p}} (K) \cdot \Omega_{-\frac{np}{n+2p}} (K^{\circ}) \geq n^{2n} \, \Omega_{p} (K) \cdot \Omega_{p} (K^{\circ}),\ \ {\hbox{if}}\ \ p \in (-\infty, -n) \cup (0, +\infty),
 \end{equation}
 and
 \begin{equation}
 \Omega_{-\frac{np}{n+2p}} (K) \cdot \Omega_{-\frac{np}{n+2p}} (K^{\circ}) \leq n^{-2n} \, \Omega_{p} (K) \cdot \Omega_{p} (K^{\circ}),\ \ {\hbox{if}}\ \ p \in (-n,0).
 \end{equation}
\end{corollary}

Finally, in a somewhat different direction,

 \begin{corollary}
 Let $K$ be a convex body of class $C^2_+$ whose Santal\'o point is at the origin and let $B_K$ be the ball in $\mathbb{R}^n$ of the same volume as $K$. Then, for any real number $p \neq -n, 0$, we have

 \begin{equation}
  \frac{\Omega_{-\frac{np}{n+2p}} (K)}{\Omega_{p}(K^{\circ})}  \geq \, \frac{\Omega_{-\frac{np}{n+2p}} (B_K)}{\Omega_{p}(B_K^{\circ})}, \ \ {\hbox{if}}\ \  \frac{p}{n+p} > 0
 \end{equation}
 and
\begin{equation}
  \frac{\Omega_{p} (K)}{\Omega_{-\frac{np}{n+2p}}(K^{\circ})}  \geq \, \frac{\Omega_{p} (B_K)}{\Omega_{-\frac{np}{n+2p}}(B_K^{\circ})}, \ \ {\hbox{if}}\ \  \frac{p}{n+p} < 0.
 \end{equation}

Equalities occur if and only if $K$ is an ellipsoid centered at the
origin.
 \end{corollary}

 \begin{proof}
 Note that $Vol (K)=Vol (B_K)$ implies, through Santal\'o's inequality, that $Vol (K^{\circ}) \leq Vol (B_K^{\circ})$. Consequently, the inequalities follow from the theorem.
 \end{proof}

 \section{New Centro-affine Invariants and Geometric Inequalities}

 We start by stating the result which introduces new centro-affine invariants. In particular, one can note that the affine invariants mentioned in the previous section appear also in a natural way.

\begin{theorem}
Let $K$ be a convex body in $\co$ evolving under (\ref{eq:fwsh}) and
denote by $K(t)$ the solution to the flow for as long as it exists
and it remains in the class $\co$.

Then  \begin{equation} Vol (K(t))= Vol (K) - t \, \Omega_{1,p} (K) +
\frac{t^2}{2}\, \Omega_{2,p} (K) - \frac{t^3}{3!} \, \Omega_{3,p}
(K) + \cdots , \label{eq:main}
\end{equation}

where each of the coefficients $\displaystyle \Omega_{k,p} (K),\
k=2, \cdots $ is a centro-affine invariant of $K$ whose integral
representation can be determined explicitly by

\begin{equation}
\Omega_{k,p} (K) = (-1)^k \frac{d^k}{dt^k} \, Vol (K(t))
{\mid}_{t=0}=(-1)^k\, {\delta}^k Vol (K). \label{eq:derivatives}
\end{equation}

\end{theorem}

\begin{proof}

The centro-affine invariance of the flow, together with the affine
invariance of the volume of Lebesgue measurable sets, here convex
bodies, implies that each of $\displaystyle \frac{d^k}{dt^k} \, Vol
(K(t)) {\mid}_{t=0}$ is a centro-affine invariant of $K$.

Let's now recall the mixed curvature function of $(n-1)$ convex
bodies $L_1, L_2, \cdots, L_{n-1}$ with support functions $h_1, h_2,
\cdots, h_{n-1}$. Denoted here by $s(h_1, h_2, \cdots, h_{n-1})$,
the mixed curvature function
 is a multi-linear, symmetric in all arguments, function on the unit sphere $\mathbb{S}^{n-1}$ which, when the bodies are sufficiently regular, is defined as an average of determinants whose columns are {\em picked} from corresponding columns of the matrices $((h_{\alpha})_{ij}+\delta_{ij} h_{\alpha})_{1 \leq i,j \leq n-1}$, $\alpha = 1 , \cdots, n-1$, until all combinations are exhausted, \cite{schneider}. The mixed curvature function of $L_1, L_2, \cdots, L_{n-1}$ integrated against the support function of any convex body $L$ of support function $h_L$ is the mixed volume

 \begin{equation}
 V(L, L_1, L_2, \cdots, L_{n-1}) = \frac{1}{n} \int_{\mathbb{S}^{n-1}} h_L \, s(h_1, h_2, \cdots, h_{n-1})\, d\mu_{\mathbb{S}^{n-1}}
 \end{equation} which is also symmetric in all of its $n$ arguments.

 In fact, $s(K_1, \ldots, K_{n-1})\, d\mu_{{\mathbb{S}}^{n-1}}$ is the unique
measure from Riesz representation theorem, representing the linear
functional on convex bodies $L \to V(L, L_1, L_2, \cdots, L_{n-1})$.
Note that, as the first description of the mixed curvature function
is purely analytical, we may extend it to all sufficiently
differentiable functions on the sphere $\mathbb{S}^{n-1}$.

 If $L_1=L_2= \cdots = L_n$, then the mixed curvature function is the curvature function of $K$. In our case, this is precisely the reciprocal of the Gauss curvature as a function on $\mathbb{S}^{n-1}$,

\begin{equation}
\frac{1}{\cu ( u)} = s(h, \cdots, h) = \det (h_{ij}+\delta_{ij} h).
\end{equation}

Consequently,

\begin{equation}
\frac{d \, Vol (K(t)}{dt}= \frac{1}{n}\,  \int_{\mathbb{S}^{n-1}}
n\, h_t (t, u) \, s(h (t, u), \cdots , h (t, u))\,
d\mu_{\mathbb{S}^{n-1}}
\end{equation}

$$= -sgn \frac{p}{n+p} \int_{\partial K} \ca^{\frac{p}{n+p}}(t,u) \, d\mu_{cK(t)}  = -sgn \frac{p}{n+p} \, \Omega_p (K (t)) =: -\Omega_{1,p} (K(t)),$$

thus, let $t=0$ and the $t$-coefficient follows by considering the
first variation of volume. In this first variation, we recover, up
to a minus sign for $p \in (-n, 0)$, the $p$-affine surface area of
$K$, while the higher order variations employs the regularity of the
boundary $\partial K$ leading to previously unknown quantities.

Having the evolution equation for the support function of $K(t)$, we
actually have the evolution equations for $\cu$ and $\ca$ obtaining
thus
 any higher
order derivative of the volume of the evolving convex body, even if
the integral expressions will become increasingly complicated.

By a direct calculation, the first genuinely new centro-affine
invariant with the simplest integral representation is

\begin{equation}
\Omega_{2,p} (K):=\left( \frac{d^2 \, Vol (K(t)}{dt^2}\right)_{t=0}=
\left( {\hbox{sgn}}\: \frac{p}{n+p} \frac{d}{dt}\, \int_{\partial K}
\ca^{\frac{p}{n+p}}(t,u) \, d\mu_{cK(t)}  \right)_{t=0} \nonumber
\end{equation}

\begin{equation}
=\frac{n(p-1)}{n+p}\, \Omega_{\frac{2np}{n-p}} (K) -
\frac{(n-1)n}{n+p}\, \int_{\mathbb{S}^{n-1}} h \ca^{\frac{p}{n+p}}
\, s (h \ca^{\frac{p}{n+p}}, h, \ldots , h)\,
d\mu_{\mathbb{S}^{n-1}}, \nonumber
\end{equation}

where, if $p=n$, the first term is  $n$ times the volume of the
polar body of $K$, coinciding with the usual definition
$\displaystyle \Omega_{\pm \infty} (K)=n \cdot Vol (K^{\circ})$.

Alternately, using the matrix notation of Lemma \ref{lemma:lin}, we
may describe $\Omega_{2,p}$ as
\begin{equation}
\Omega_{2,p} (K):= \frac{n(p-1)}{n+p}\, \Omega_{\frac{2np}{n-p}} (K)
\hspace{5.25cm}
\end{equation}

\begin{equation} \hspace{1.75cm} -\frac{n}{n+p}\, \int_{\mathbb{S}^{n-1}}
h \ca^{\frac{p}{n+p}} \, \sum_{i,j=1}^n \frac{\partial S}{\partial
a_{ij}} \: \left(\left({h \ca^{\frac{p}{n+p}}}\right)_{ij} +
\delta_{ij} \, h \ca^{\frac{p}{n+p}} \right) \,
d\mu_{\mathbb{S}^{n-1}} \nonumber
\end{equation}

\begin{equation} =\frac{n(p-1)}{n+p}\, \Omega_{\frac{2np}{n-p}} (K) -\frac{n}{n+p}\, \int_{\mathbb{S}^{n-1}}
h \ca^{\frac{p}{n+p}} \, \left({\tilde{S}^{-1}} \right)^{ij} \:  S
\: \left(\left({h \ca^{\frac{p}{n+p}}}\right)_{ij} + \delta_{ij} \,
h \ca^{\frac{p}{n+p}} \right) \, d\mu_{\mathbb{S}^{n-1}} \nonumber
\end{equation}

\begin{equation} =\frac{n(p-1)}{n+p}\, \Omega_{\frac{2np}{n-p}} (K) -\frac{n}{n+p}\, \int_{\mathbb{S}^{n-1}}
h \ca^{\frac{p}{n+p}} \, \left({\tilde{S}^{-1}} \right)^{ij}  \:
\left(\left({h \ca^{\frac{p}{n+p}}}\right)_{ij} + \delta_{ij} \, h
\ca^{\frac{p}{n+p}} \right) \frac{1}{\cu} \,
d\mu_{\mathbb{S}^{n-1}}, \nonumber
\end{equation}

using Einstein's summation convention in the last two lines.

 We found the first description to be more convenient for the derivation of an isoperimetric inequality involving $\Omega_{2,p}$, while the concavity of $S^{1/(n-1)}$ with respect to its entries (see \cite{Mi}) favors the passage to the next variation of volume and obtaining an inequality involving $\Omega_{3,p}$.

\end{proof}

\begin{proposition}[Isoperimetric Inequality for $\Omega_2,p$]
For any $p \geq 1$ and any $K \in \co$,
\begin{equation}
\Omega_{2,p} (K) \geq \frac{p-n}{n+p} \, \frac{\Omega_p^2 (K)}{Vol
(K)}, \label{eq:iso2}
\end{equation}
with equality if and only if $K$ is an ellipsoid.
\end{proposition}

\begin{proof}
Note that H\"older's inequality implies

\begin{equation}
\Omega_{\frac{2np}{n-p}} (K) \cdot Vol (K) \geq \frac{1}{n}\,
\Omega_p^2 (K),
\end{equation}

with equality (in the class $\co$) if and only if $\ca$ is constant
on $\mathbb{S}^{n-1}$, hence $K$ is an ellipsoid. This takes care of
the first term of $\Omega_{2,p} (K)$.

For the second term, suppose that $\displaystyle h
\ca^{\frac{p}{n+p}} : \mathbb{S}^{n-1} \to \mathbb{R}$ is a support
function of a convex body called, say, $\tilde{K}$, in which case we
will call the body $K$ $p$-elliptic (for $p=1$ this coincides with
the usual definition of ellipticity, see, for example,
\cite{leichtweiss}). Then $ \displaystyle  \frac{1}{n}
\int_{\mathbb{S}^{n-1}} h \ca^{\frac{p}{n+p}} \, s (h
\ca^{\frac{p}{n+p}}, h, \ldots , h)\, d\mu_{\mathbb{S}^{n-1}}$, is
the mixed volume $V(\tilde{K}, \tilde{K}, K, \ldots , K)$, hence the
following Minkowski inequality can  be used  \begin{equation}
V^2(\tilde{K}, K, K, \ldots, K) \geq V(\tilde{K}, \tilde{K}, K,
\ldots , K) \cdot Vol (K),
\end{equation}
with equality if and only if $K$ is homothetic to $\tilde{K}$,
\cite{leichtweiss}.
However, if there exists $\lambda$, positive real number, and there
exists ${\bf{x}}_0 \in \mathbb{R}^n$ such that $\displaystyle h (u)
\ca^{\frac{p}{n+p}}(u) = \lambda h (u) - <{\bf{x}}_0, u>$, for all
$\displaystyle u \in \mathbb{S}^{n-1}$, then $\displaystyle
\ca^{\frac{p}{n+p}}(u) = \lambda  - <{\bf{x}}_0, u>\frac{1}{h(u)}$
and the centro-affine invariance of $\ca$ implies ${\bf{x}}_0 =
{\bf{0}}$. Thus $\ca$ must be constant and $K$ an ellipsoid centered
at the origin.

The proof is then concluded by
\begin{equation}
\Omega_{2,p} (K) \geq \frac{p-1}{n+p}\, \frac{\Omega_p^2 (K)}{Vol
(K)} - \frac{n-1}{n+p}\, \frac{\Omega_p^2 (K)}{Vol (K)}.
\end{equation}
Before we proceed further, note that the restriction on the range of
$p$ comes solely from the sign of the coefficients of the two terms
forming $\Omega_{2,p}$.

It remains to discuss the case when the function $\displaystyle h
\ca^{\frac{p}{n+p}} $ defined on the unit sphere is not a support
function of a convex body.

We claim that for $c$, a large enough constant, the function  $\psi
: \mathbb{S}^{n-1} \to \infty$  defined by $\displaystyle \psi := h
\ca^{\frac{p}{n+p}} + c\, h$ is the support function of a convex
body. To prove this, since all functions involved here are smooth,
extend $\psi$ to $\mathbb{R}^n \setminus \{0\}$ by $\psi (x)= |x| \,
\psi (x/|x|)$ and consider the Hessian of $\psi$ on
$\mathbb{R}^{n}$. Using the properties of determinants, we write the
Hessian  as a polynomial in $c$, $\displaystyle {\hbox{Hess}}\,
(\psi)=\sum_{k=1}^{n} c^k \; H_k$, where each $H_k,\ k=0, \ldots,
n$, is a Hessian like matrix and, in particular, $H_{n}$ is the
Hessian of the support function of $K$ hence positive definite due
to the strict convexity of the body. Therefore as $c$ goes to
infinity, ${\hbox{Hess}}\, (\psi)$ will become eventually positive
definite too. Thus, we may pick a constant $c$ large enough so that
the value of Hess $\psi$ is strictly positive, thus $\psi$ is
convex, and the support function of a convex body. We denote by
$\tilde{K}$ the convex body of support function $\psi$ and revert to
the restriction of $\psi$ on the unit sphere.

In fact, we will conclude the proof with the following lemma which
we believe to be known, but we could not find a reference.

\begin{lemma}
Let $K$ be a convex body of class $C^2_+$ and let $\psi$ be a twice
differentiable function on the unit sphere. Then

\begin{equation}
V^2 (\psi, K, K, \ldots, K) \geq V(\psi, \psi, K, \ldots, K) \cdot
V(K, K, K, \ldots, K).
\end{equation}
\end{lemma}

 As, for some constant $c$, $\psi + ch$ is a support function of a convex body $\tilde{K}$, we may apply the previous Minkowski inequality for $\tilde{K}$ and obtain, due to the multilinearity and the symmetry of the curvature function,

 \begin{eqnarray} &0& \leq V^2(\tilde{K}, K, K, \ldots, K) -Vol (K) \cdot V(\tilde{K}, \tilde{K}, K, \ldots, K) \nonumber
 \\ & = &\left(\frac{1}{n} \int_{\mathbb{S}^{n-1}} \psi \, s(h, h, \ldots, h) \, d\mu_{\mathbb{S}^{n-1}} \right)^2 -Vol (K) \cdot \frac{1}{n}  \int_{{\mathbb{S}}^{n-1}} \psi \, s (
\psi, h, \ldots, h)\, d\mu_{{\mathbb{S}}^{n-1}} \nonumber  \\ &=
&\left[ \frac{1}{n^2} \left(\Omega_p(K) \right)^2 - Vol (K) \cdot
\frac{1}{n} \cdot \int_{{\mathbb{S}}^{n-1}} h \ca^{\frac{p}{n+p}} \,
s ( h \ca^{\frac{p}{n+p}} , h, \ldots, h)\,
d\mu_{{\mathbb{S}}^{n-1}} \right] \nonumber \\ &+&   c\, Vol (K) \,
\frac{2}{n}   \int_{{\mathbb{S}}^{n-1}} h \ca^{\frac{p}{n+p}}\, s(h,
h, \ldots, h) \, d \mu_{{\mathbb{S}}^{n-1}} + c^2 Vol^2 (K)
\nonumber \\ &-&Vol (K) \, \left[  \frac{2c}{n}
\int_{{\mathbb{S}}^{n-1}}h \ca^{\frac{p}{n+p}} s(h, h, \ldots, h) \,
d\mu_{{\mathbb{S}}^{n-1}} - \frac{c^2}{n} \int_{{\mathbb{S}}^{n-1}}
h\, s(h , h, \ldots,
h) \, d\mu_{{\mathbb{S}}^{n-1}}  \right] \nonumber \\
&=& \left[ \frac{1}{n^2} \left(\Omega_p(K) \right)^2 -Vol (K) \cdot
\frac{1}{n}  \int_{{\mathbb{S}}^{n-1}} h \ca^{\frac{p}{n+p}} \, s (h
\ca^{\frac{p}{n+p}}, h, \ldots, h)\, d\mu_{{\mathbb{S}}^{n-1}}
\right]. \nonumber
\end{eqnarray}

\medskip

This ends the proof of the lemma. We now apply the two inequalities
as before to conclude the proof of the isoperimetric inequality
while the equality case occurs again if $\ca$ is constant on the
sphere which corresponds to $K$ being an ellipsoid.

\end{proof}

\begin{corollary}
Let $K$ be an arbitrary convex body in $ \co$. Denote by $K^{\circ}$
its polar with respect to the origin, $\displaystyle K^{\circ} := \{
x \in \mathbb{R}^n \mid x \cdot y \leq 1, \ \forall y \in K \},$ and
by $\ca^{\circ}$ the centro-affine curvature of $K^{\circ}$ as a
function on the unit sphere.

Then the following inequalities hold

\begin{equation}
 \left[\frac{1}{n\: Vol (K)} \int_{\partial K} \ca^2 \, d\mu_{cK} \right]^{-1/2} \hspace{-0.35cm} \leq \frac{Vol (K)}{Vol (K^{\circ})} \leq
 \left[\frac{1}{n\: Vol (K^{\circ})}  \int_{\partial K^{\circ}} (\ca^{\circ})^2 \, d\mu_{cK^{\circ}} \right]^{1/2} \nonumber
\end{equation}

\smallskip

and the equalities occur simultaneously when $K$ is an ellipsoid.
\end{corollary}

\begin{proof}
Since $\mathbb{S}^{n-1}$ is a compact manifold and all functions
involved are smooth, we apply Lebesgue dominated convergence theorem
to the integrals in the expression of $\Omega_p (K)$ and
$\Omega_{2,p} (K)$ respectively.

From (\ref{eq:iso2}), we obtain as $p \to \infty$ that

\begin{equation}
n\, \Omega_{-2n} (K) \geq \frac{n^2\, Vol^2 (K^{\circ})}{Vol (K)}
\label{eq:coro}
\end{equation}

which is, after a re-arrangement of the factors, the first
inequality.

Simultaneously, we may apply the inequality (\ref{eq:coro}) to
$K^{\circ}$ and obtain the second inequality.
\end{proof}


\begin{lemma}
As $K$ evolves by the centro-affine curvature flow (\ref{eq:fwsh}),
the volume of the dual body $K^{\circ}$ changes by

\begin{equation}
\frac{d}{dt} Vol (K^{\circ}) = {\hbox{sgn}} \frac{p}{n+p} \,
\Omega_{-\frac{np}{n+2p}} (K^{\circ}).
\end{equation}

\end{lemma}

\begin{proof}
The formula follows from a direct calculation as

\smallskip

$\displaystyle Vol (K^{\circ}) = \frac{1}{n} \,
\int_{\mathbb{S}^{n-1}} \frac{1}{h^n}\, d\mu_{\mathbb{S}^{n-1}}$
and, as $K, K^{\circ}$ smooth, $\Omega_q (K)= \Omega_{\frac{n^2}{q}}
(K^{\circ})$ for any $q$.

\end{proof}

Compared with the rate of change of the volume of a convex body $L$
whose boundary is deformed by a normal vector field with speed $v$
(as a function of $u$), which is $\displaystyle \frac{d}{dt} Vol
(L)=\int_{\mathbb{S}^{n-1}} v \, \frac{1}{\cu_L}\,
d\mu_{\mathbb{S}^{n-1}}$
 we infer that while $K$ evolves under

 \begin{equation} \left\{  \begin{alignedat}{2}
\frac{\partial h(u,t)}{\partial t}
&=&-{\hbox{sgn}}\left({\frac{p}{n+p}}\right)\: h(u,t)
 \: {\ca}^{\frac{p}{n+p}} (u,t)
\\
h(u,0)&=&h_K(u), \hspace{4.5cm}  \end{alignedat}\right.
\label{eq:fdirect} \end{equation}

which is precisely (\ref{eq:fwsh}),  its dual $K^{\circ}$ evolves
under the flow

 \begin{equation} \left\{  \begin{alignedat}{2}
\frac{\partial h(u,t)}{\partial t}
&=&{\hbox{sgn}}\left({\frac{p}{n+p}}\right)\: h(u,t) \:
{\ca}^{-\frac{p}{n+p}} (u, t)
\\
h(u,0)&=&h_{K^{\circ}}(u). \hspace{4.5cm}  \end{alignedat}\right.
\label{eq:fdual} \end{equation}

We call the latter the {\em dual flow} and note that is, in fact,
the {\em direct flow} for $\displaystyle
p^{\circ}=-\frac{np}{n+2p}$. Hence, for $p^{\circ} \geq 1$, we can
extend the isoperimetric inequality to a negative range.

\begin{corollary}
For any $K \in \co$, the isoperimetric inequality for $\Omega_2,p$
(\ref{eq:iso2}) holds also for any $\displaystyle p \in \left( -
\frac{n}{2}, - \frac{n}{n+2} \right]$ with equality if and only if
$K$ is an ellipsoid.
\end{corollary}

The isoperimetric inequality for $\Omega_2,p$, generalizes
Proposition 2.2 in \cite{stancu_werner} and we believe that the
approach we had there can be extended to define iteratively the
newly introduced $\Omega_{k, p}$ for arbitrary convex bodies via
weighted floating bodies. Implicitly, that will imply that
$\Omega_{k, p}(P)$ is either zero or infinite on polytopes.

\section{Geometric Interpretation of Ludwig-Reitzner
Characterization of Affine Invariants}

We start by recalling the strongest result up to date on the
classification of centro-affine invariants of convex bodies due to
M. Ludwig and M. Reitzner.

\begin{theorem} [\cite{Ludwig-Reitzner}]
A real valued functional $\Phi$ on the space $\mathcal{K}_0^n$ of
convex bodies in $\mathbb{R}^n$ containing the origin in their
interior is an $SL(n)$-invariant, upper semicontinuous valuation
that vanishes on polytopes  if and only if there exists $\phi : [0,
\infty) \to [0, \infty)$ concave, with $\displaystyle \lim_{t \to
0^+} \phi (t)=\lim_{t \to \infty} \phi (t)/t =0$, such that
\begin{equation} \Phi (K)=\int_{\partial K} \phi (\ca)\, d\mu_{cK},\
\ \ \forall K \in \mathcal{K}_0^n.
\end{equation}

\end{theorem}

The authors called $\displaystyle \Phi (K)=: \Omega_{\phi} (K)$ the
$L_{\phi}$-affine surface area of $K$ and, indeed, one can easily
see the generalization from $\Omega_p(K)$.

\begin{theorem}
For any given smooth, concave $\phi : [0, \infty) \to [0, \infty)$,
there exists an affine invariant flow on $\co$ such that, for any $K
\in \co$,
\begin{equation}
\Phi (K)= \lim_{t \searrow 0} \frac{Vol (K)-Vol (K(t))}{t},
\label{eq:FI}
\end{equation}

where $K(t)$ is the solution to the flow with initial data $K$ at
time $t >0$.
\end{theorem}

\begin{proof} Consider the centro-affine invariant flow
\begin{eqnarray}
\frac{\partial X(u,t)}{\partial t} =-\ca^{-\frac{1}{n+1}} (u,t)\,
\phi (\ca)  \: \mathcal{N}(u,t)
\label{eq:fws} \\
X(u,0)=X_K(u), \hspace{2cm} \nonumber
\end{eqnarray}

which reduces to the scalar equation

\begin{eqnarray}
\frac{\partial h(u,t)}{\partial t} =-h (u,t) \phi (\ca (u,t))
\label{eq:phih} \\
h(u,0)=h_K(u). \hspace{2cm} \nonumber
\end{eqnarray}

The concavity of $\phi$ insures that the Cauchy problem is strictly
parabolic in $\co$, thus short time existence and uniqueness of
solutions follow from the classical theory.

It follows immediately that \begin{equation} \frac{d}{dt} Vol
(K(t))= - \int_{\partial K} \phi (\ca)\, d\mu_{cK},
\end{equation}

implying (\ref{eq:FI}).
\end{proof}

Note that Ludwig-Reitzner hypotheses on $\phi$ are not all needed
for the geometric interpretation of $\Omega_{\phi}$ of convex bodies
whose boundary is sufficiently regular. In fact, we will state below
a version of the previous theorem for $C^2_+$ convex bodies which
contain the origin which relies only on the positivity of $\phi$.
Since the boundary of $K$ is now only of class $C^2$, we must
restrain from using the flow technique. Yet, the following lemma was
first derived by considering a flow analogous to (\ref{eq:phih}) in
a smooth case, then relaxing the set up.

\begin{theorem}
For any function $\phi : (0, \infty) \to (0, \infty)$ and for any $K
\in C^2_+$, convex body with $C^2$ boundary and everywhere positive
Gauss curvature, such that the origin belongs to the interior of
$K$, there exists a family of convex bodies parameterized by $t$,
$K_{\phi}(t)$, such that
\begin{equation}
\Phi (K)= \lim_{t \searrow 0} \frac{Vol
(K_{\phi}^{\circ}(t))-Vol(K^{\circ})}{t}, \label{eq:FLI}
\end{equation}
where $K^{\circ}$ represents the polar body of $K$ with respect to
the origin, $\displaystyle K^{\circ}=\{ y \in \mathbb{R}^n \mid x
\cdot y \leq 1 \ \forall x \in K \}$. \label{theorem:polar}
\end{theorem}

\begin{proof}
We will recall first the definition of the convex floating body
which belongs to Sch\"{u}tt and Werner, \cite{SW1}.

Let $K$ be a convex body in ${\mathbb{R}}^{n}$ with support function
$h_K:{\mathbb{S}}^{n-1} \rightarrow {\mathbb{R}}$. For each unitary
direction ${{u}} \in {\mathbb{S}}^{n-1}$, there exists a unique
hyperplane of normal ${{u}}$ supporting the boundary of $K$, $\ds
{\it{H}}_{{{u}}} = \left\{ y \in {\mathbb{R}}^{n} \mid {{u}} \cdot y
= h_K({{u}}) \right\}. $

If ${\it{H}}_{{{u}}, \delta} = \left\{ y \in {\mathbb{R}}^{n} \mid
{{u}} \cdot y  = h_{K_\delta}({{u}}) \right\}$ denotes the
hyperplane parallel to ${\it{H}}_{{{u}}}$,
 such that the $n$-dimensional volume of the cap cut from K by ${\it{H}}_{{\bf{u}},
 \delta}$ is $\delta$,
\begin{equation}
Vol (\left\{ y \in K \mid h_{K_\delta} ({{u}}) \leq {{u}} \cdot y
\leq h_K({{u}}) \right\}  ) = \delta,
\end{equation}

for some positive $\delta < Vol (K)/2$, then \begin{equation}
\displaystyle K_{\delta}= \bigcap_{{{u}} \in {\mathbb{S}}^{n-1}}
\left\{ y \in {\mathbb{R}}^{n} \mid {{u}} \cdot y \leq
h_{K_{\delta}}({{u}}) \right\}
\end{equation}

is said to be the {\em convex floating body} of $K$ of factor
$\delta$.

Recall that we required the origin to be in $Int (K)$. For each
fixed function $\phi$ as above, we will now modify the definition of
the convex floating body by allowing $\delta$ to vary continuously
with respect to the normal direction $u$. Let $t>0$ be a fixed real
number. Then there exists an interval $(0, t_0^K) \subset (0, (Vol
(K)/2)^{2/(n+1)})$ such that each hyperplane ${\it{H}}_{{{u}},
\delta}$ cuts off a volume equal to $\delta_{\phi, u}:=\displaystyle
\frac{n+1}{\omega_{n-1}} \cdot \left[ \frac{t}{ 2}\, \phi (\ca
(u))\, \ca^{-\frac{n+2}{n+1}}(u)\right]^{(n+1)/2}, $ where
$\omega_{n-1}= Vol (B_2^{n-1})$, and we define

 \begin{equation}
\displaystyle K_{\phi} (t)= \bigcap_{{{u}} \in {\mathbb{S}}^{n-1}}
\left\{ y \in {\mathbb{R}}^{n} \mid {{u}} \cdot y \leq
h_{K_{\delta_{\phi, u}}}({{u}}) \right\}.
\end{equation}

An immediate consequence of the definition is that $K_{\phi} (t)$ is
non-empty and convex for all $t$ in the interval $(0, t_0^K)$, an
interval which is possibly very small.

 Following our techniques from \cite{stancu}, we have that the support function of $K_{\phi} (t)$ satisfies

 \begin{equation}
 h_{K_\phi} (t, u)=h_K(u)-t\, {\ca^{-\frac{n+2}{n+1}}(t, u)}{\phi (\ca (t,u))}\, \cu^{\frac{1}{n+1}}(t,u)  + o(t), \label{eq:suport}
 \end{equation}
for all $\displaystyle u \in \mathbb{S}^{n-1}$. Then, the claim of
the theorem follows from

 \begin{eqnarray}
Vol (K_{\phi}(t)^{\circ})&=& \frac{1}{n} \int_{\mathbb{S}^{n-1}}
(h_{K_\phi})^{-n}(t,u)
\, d\mu_{\mathbb{S}^{n-1}}(u)  \nonumber \\
& = &  \frac{1}{n} \int_{\mathbb{S}^{n-1}} \left( h_K-t\, {\ca^{-\frac{n+2}{n+1}}}{\phi (\ca)}\, \cu^{\frac{1}{n+1}}  + o(t) \right)^{-n}\, d\mu_{\mathbb{S}^{n-1}}  \nonumber \\
& = &  \frac{1}{n} \int_{\mathbb{S}^{n-1}}  h_K^{-n}\, \left( 1-t\, {\ca}^{-1}{\phi (\ca)}\, + o(t) \right)^{-n}\, d\mu_{\mathbb{S}^{n-1}}   \\
 &= & Vol (K^{\circ})+t\,
\int_{\partial K} \phi (\ca)\, d\mu_{cK} +o(t). \nonumber
\end{eqnarray}

 \end{proof}

 \begin{remark}
 Note that the coefficient of $t$ in (\ref{eq:suport}) is nothing else but $\displaystyle
  h_K(t,u)\, \frac{\phi (\ca (t,u))}{\ca (t,u)}$. Therefore the condition $\displaystyle \lim_{t \to \infty} \phi (t)/t =0$, allows to extend the definition of $K_{\phi}$ to the case when the origin lies on the boundary of $K$.
 \end{remark}

\bigskip

\end{document}